\documentclass{amsart}
\usepackage{latexsym,amssymb}
\usepackage[english]{babel}
\usepackage{amsfonts}

\def\Z{\mathbb{Z}}

\def\Kmn{K_{m\times n}}
\def\B{\mathcal{B}}
\def\G{\Gamma}
\def\m{\overline{m}}
\def\ccup{\mathop{\cup}}

\newtheorem{defini}{Definition}[section]
\newtheorem{prop}[defini]{Proposition}
\newtheorem{lem}[defini]{Lemma}

\newtheorem{cor}[defini]{Corollary}
\newtheorem{thm}[defini]{Theorem}

\theoremstyle{definition}
\newtheorem{rem}[defini]{Remark}
\newtheorem{ex}[defini]{Example}

\begin{document}

\title[Cyclic HCS of the complete multipartite graph]{Cyclic
hamiltonian cycle systems \\ of the complete multipartite graph: \\
even number of parts}

\author{Francesca Merola}
\address{Dipartimento di Matematica e Fisica, Universit\`a Roma Tre, Largo S.L. Murialdo 1, I-00146 Roma, Italy}
\email{merola@mat.uniroma3.it}

\author{Anita Pasotti}
\address{DICATAM - Sez. Matematica, Universit\`a degli Studi di
Brescia, Via
Branze 43, I-25123 Brescia, Italy}
\email{anita.pasotti@unibs.it}

\author{Marco Antonio Pellegrini}
\address{Dipartimento di Matematica e Fisica, Universit\`a Cattolica del Sacro Cuore, Via
Musei 41,
I-25121 Brescia, Italy}
\email{marcoantonio.pellegrini@unicatt.it}

\begin{abstract}
A hamiltonian cycle system (HCS, for short) of a graph $\Gamma$  is a partition of the edges of
$\Gamma$ into hamiltonian cycles.
A HCS is \emph{cyclic} when it is invariant under a cyclic permutation of all the vertices of
$\Gamma$; the existence problem for a cyclic HCS has been completely solved by Buratti and Del Fra
in 2004
when
$\Gamma$ is the complete graph $K_v$, $v$ odd,
 and by Jordon and Morris in 2008 when $\Gamma$ is
the complete graph minus a $1$-factor $K_v-I$, $v$ even. In this work we present a complete solution
to  the existence problem of a cyclic HCS for  $\Gamma = \Kmn$, the complete multipartite graph,
when the number of parts $m$ is even.
We also give necessary and sufficient conditions for the existence of a cyclic and
\emph{symmetric} HCS of $\Gamma$;
the notion of a symmetric HCS  of a graph
$\Gamma$ has been introduced in 2004 by Akiyama, Kobayashi, and Nakamura for   $\Gamma =
K_v$, $v$ odd, in 2011 by Brualdi and Schroeder when $\Gamma = K_v-I$, $v$ even, and,
very recently, by
Schroeder when $\Gamma$ is the complete multipartite graph.
\end{abstract}


\maketitle

\section{Introduction}

Throughout this paper, $K_v$ will denote the complete graph on $v$ vertices
and, if $v$ is even, $K_v-I$ will denote the cocktail party graph of order $v$, namely the graph obtained from $K_v$ by removing a
$1$-factor $I$, that is a set of $\frac{v}{2}$ pairwise disjoint edges.
Also
$\Kmn$  will denote
the complete multipartite graph with $m$ parts of size $n$; if $n=1$,
we may identify $K_{m\times 1}$ with  $K_m$,
while if $n=2$,
$K_{m\times 2}$ is nothing but the cocktail party graph $K_{2m}-I$.

For any graph $\Gamma$
we write $V(\G)$ for the set of its vertices and $E(\G)$ for the set of its edges.
We denote by $C_{\ell}=(c_0,c_1,\ldots,c_{\ell-1})$ the cycle of length $\ell$ whose edges are
$[c_0,c_1],[c_1,c_2],\ldots,[c_{\ell-1},c_0]$.
An $\ell$-\emph{cycle system} of a graph $\G$ is a set $\B$ of cycles of length $\ell$, said \emph{blocks}, whose
edges partition $E(\G)$; clearly a graph may admit a cycle system only if the degree of each vertex is even.
For general background on cycle systems
we refer to the survey \cite{BrRo}.
An $\ell$-cycle system $\B$ of $\G$ is said to be
 \emph{hamiltonian} if $\ell=|V(\G)|$, namely if each cycle of the system
passes through all the vertices of $\G$, and it is said to be
\emph{cyclic}
if we may identify $V(\Gamma)$ with the cyclic group $\Z_v$, and then for any $(c_0,c_1,\ldots,c_{\ell-1})\in\B$, we have also
$(c_0+1,c_1+1,\ldots,c_{\ell-1}+1)\in\B$.
The existence problem for cyclic cycle systems of $K_v$ has generated a considerable amount of interest. Many authors
have contributed to give
 a complete answer in the case $v\equiv 1$ or $\ell\pmod {2\ell}$,
see \cite{BDF2003,BDF2004,Pe,Rosa1,Rosa2,Rosa3,Vietri}.
We point out in particular that the existence problem of a cyclic hamiltonian cycle system for $K_v$ has been solved by Buratti and Del Fra in \cite{BDF2004}, and that for $K_v-I$ the existence problem of a cyclic hamiltonian cycle system has been solved by Jordon and Morris \cite{JM}.

The existence problem for cycle systems of the complete
multipartite graph has not
been solved yet, but we have many interesting recent results on this topic (see for instance \cite{BCS1, BCS2, SC1,SC2}).
Still, very little is known about the same problem with
the additional constraint that the system be cyclic.
We have a complete solution in the following very special cases:
the length of the cycles is equal to the size of the parts \cite{BuG};
the cycles are Hamiltonian and the parts have size two \cite{BMnew,JM}.
We have also some partial results in \cite{BePa}.

Hamiltonian cycle systems (briefly HCS) of $\Kmn$ have been shown to exists  (\cite{LA}) whenever the degree of each vertex of the graph, that is $(m-1)n$, is even; in this paper we  start investigating cyclic HCS  of $\Kmn$, and completely solve the existence problem for complete multipartite graphs with an even number of parts.

We also consider the existence of a \emph{symmetric} HCS for $\Kmn$, a concept recently introduced by
Schroeder in \cite{schr} generalizing the notion of symmetry given in \cite{BS} for cocktail
party graphs: in this definition, a HCS for  $\Kmn$ is \emph{$\phi_n$-symmetric} if each cycle in the system is invariant under a fixed-point-free automorphism of order $n$. We will show that the cycle systems we shall construct in will turn out to be symmetric in this sense.

The paper is organized as follows: after some preliminary notes in Section \ref{se:prel} on the methods we shall use,
in Section \ref{se:ne}
we establish a necessary condition in the case $n$ even for the existence of a cyclic cycle system  (not necessarily hamiltonian) of $\Kmn$
from which we derive a necessary condition for the existence of a cyclic HCS of $\Kmn$.
Then in Section \ref{se:ex} we give a complete solution to the existence problem of a cyclic HCS with an even number of parts,
proving that in
this case the necessary condition we found is also sufficient.
The main result of this paper is the following.
\begin{thm}\label{th:main}
Let $m$ be even; a cyclic and $\phi_n$-symmetric HCS for $\Kmn$ exists if and only if
\begin{itemize}
\item $n$ is even, and
\item if $n\equiv 2 \pmod 4$, then $m\equiv 2 \pmod 4$.
\end{itemize}
\end{thm}

The proof of Theorem \ref{th:main} will follow from the various results proved in Sections \ref{se:ne} and \ref{se:ex}.
First, in Corollary \ref{cor:ne} we give the necessary condition for the existence of a cyclic HCS of $\Kmn$.
Then, in Proposition \ref{prop:bipartito} we study the bipartite case, finally in Theorem \ref{thm:m_pari}
and in Theorem \ref{thm:mn2}
 we deal with the case $n\equiv 0\pmod 4$,
and $n\equiv 2\pmod 4$, respectively.

\section{Preliminaries}\label{se:prel}

The main results of this paper will be obtained by using the method of \emph{partial differences}
introduced by Marco Buratti in 2004 and used in many papers, see for instance
\cite{BePaGC, B, BDF2003, BDF2004, BM, BMnew, BR, Vietri}. Here we recall some definitions  and results useful in the rest of the paper.

\begin{defini}
Let  $C=(c_0,c_1,\ldots,c_{\ell-1})$ be an $\ell$-cycle with
vertices in an abelian group $G$ and let $d$ be the order of the stabilizer of $C$ under
the natural action of $G$, that is $d=|\{g\in G : C+g=C\}|$.
The multisets
\begin{eqnarray*}
\Delta C &=& \{\pm(c_{h+1}-c_{h})\ |\ 0\leq h < \ell\},\\
\partial C &=& \{\pm(c_{h+1}-c_{h})\ |\ 0\leq h < \ell/d\},
\end{eqnarray*}
where   the subscripts are taken modulo $\ell$,
are called the \emph{list of differences} from $C$ and
the \emph{list of partial differences} from $C$, respectively.
\end{defini}

More generally, given a set $\B$ of $\ell$-cycles with vertices
in $G$, by $\Delta \B$ and $\partial \B$ one means the union (counting
multiplicities) of all multisets $\Delta C$ and $\partial C$ respectively,
where  $C\in \B$.

We recall that the \emph{Cayley graph on a group $G$ with connection set $\Omega$}, denoted by
$Cay[G:\Omega]$,
is the graph whose vertices are the elements of $G$ and in which two vertices are adjacent if
and only if their difference is an element of $\Omega$.
Note that $\Kmn$ can be interpreted as the Cayley graph $Cay[\Z_{mn}:\Z_{mn}-m\Z_{mn}]$,
where by  $m\Z_{mn}$ we mean the subgroup of order $n$ of $\Z_{mn}$. The vertices of $\Kmn$
will be always understood as elements of $\Z_{mn}$ and the parts of
$\Kmn$ are the cosets of   $m\Z_{mn}$ in $\Z_{mn}$.
The stabilizer and the orbit of any subgraph $\G$ of $\Kmn$ will be understood to be under the natural
action of $\Z_{mn}$ and will be denoted by $Stab(\G)$ and $Orb(\G)$, respectively.
A cyclic HCS of $\Kmn$ is completely determined by 
a set of \emph{base cycles}, namely a complete system of representatives
for the orbits of its cycles under the action of $\Z_{mn}$. A set of base cycles may be obtained using
partial differences, as the next theorem, which
adapts a result contained in \cite{BP}, shows.

\begin{thm}\label{thm:basecycles}
A set $\B$ of $mn$-cycles is a set of base cycles of a cyclic HCS of $\Kmn$
if and only if $\partial \B=\Z_{mn}-m\Z_{mn}$.
\end{thm}

In Example \ref{ex:10x6} we will show how to construct a cyclic HCS of $K_{10\times 6}$  applying
Theorem \ref{thm:basecycles}.

 For our purposes the following notation will be useful.\\
Let $c_0,c_1,\ldots,c_{r-1},x$ be elements of an additive  group $G$, with $x$ of order $d$.
 The closed trail represented by  the concatenation of the sequences
\begin{center}\begin{tabular}{c}
$[c_0,c_1,\ldots,c_{r-1}]$,\cr
$[c_0+x,c_1+x,\ldots,c_{r-1}+x]$,\cr
$[c_0+2x,c_1+2x,\ldots,c_{r-1}+2x]$,\cr
$\cdots$\cr
$[c_0+(d-1)x,c_1+(d-1)x,\ldots,c_{r-1}+(d-1)x]$\cr
\end{tabular}
\end{center}
will be denoted by
\begin{equation}\label{closedtrail}
[c_0,c_1,\ldots,c_{r-1}]_x.
\end{equation}
For brevity, given $P=[c_0,c_1,\ldots,c_{r-1}]$, we write $[P]_x$ for the closed trail $[c_0,c_1,\ldots,$ $c_{r-1}]_x$.

\begin{rem}\label{rem}
Note that $[c_0,c_1,\ldots,c_{r-1}]_x$ is a $(dr)$-cycle if and only
if the elements $c_i$, for $i=0,\ldots,r-1$, belong to pairwise
distinct cosets of the subgroup $\langle x \rangle$ in $G$. Also,  if
$C=[c_0,c_1,\ldots,c_{r-1}]_x$ is a $(dr)$-cycle then
$$\partial C=\{\pm(c_i-c_{i-1})\ | \ i=1,\ldots,r-1\}\cup\{\pm(c_0+x-c_{r-1})\}.$$
We point out that  in the case of cyclic HCS of $\Kmn$,
we have that $dr=mn$, the order of $Stab(C)$ is $d$ and the length of the $\Z_{mn}$-orbit of $C$ is $r$.
\end{rem}

\begin{ex}\label{ex:10x6}
Here we present the construction of a cyclic 
HCS of $K_{10\times 6}$.
Consider the following cycles with vertices in $\Z_{60}$:
$$
C_1  = [0,19,1,17,3,15,6,14,8,12]_{10},\quad
C_2  = [0,29,1,28,2,27,3,26,4,25]_{10},$$
$$C_3=[0,3]_2,\quad C_4=[0,7]_2, \quad C_5=[0,13]_2,\quad C_6=[0]_{17}.$$
One can easily check that $\B=\{C_1,\ldots,C_6\}$ is a set of hamiltonian cycles of $K_{10\times 6}$
and that:
\begin{eqnarray*}
\partial C_1 & = &  \pm\{19,18,16,14,12,9,8,6,4,2\},\\
\partial C_2 & =& \pm \{29,28,27,26,25,24,23,22,21,15\},
\end{eqnarray*}
$$\partial C_3=\pm \{3,1\},\quad \partial C_4=\pm \{7,5\},\quad \partial C_5=\pm \{13,11\},\quad \partial C_6=\pm \{17\}.$$
Hence $\partial \B=\Z_{60}-10\Z_{60}$. So, in view of Theorem \ref{thm:basecycles},
we can conclude that $\B$ is a set of base cycles of a cyclic HCS of $K_{10\times 6}$.\\
Explicitly the required system consists of  the following 27 cycles:
$$\{C_1+i, C_2+i \mid  i=0,\ldots, 9\}\cup \{C_3+i, C_4+i, C_5+i\mid  i=0,1\} \cup \{C_6\}.$$
\end{ex}

A HCS of the complete graph $K_v$, $v$ odd, is said to be {\em symmetric}
if there is an involutory permutation $\phi$ of the vertices of $K_v$ fixing all its cycles;
in the case $v$ even, a HCS of the cocktail party  graph $K_v-I$ is {\em symmetric}
if all its cycles are fixed by the involution  switching all pairs of endpoints
of the edges of $I$. This definition is due to Akiyama, Kobayashi and Nakamura \cite{AKN} in the case $v$ odd, and to
Brualdi and Schroeder \cite{BS} in  the case $v$ even; it is easy to see that a symmetric HCS always exists in the odd case (an example is the well-known Walecki construction), while in the even case we have the following result.

\begin{thm}[Brualdi and Schroeder \cite{BS}]\label{BS}
A symmetric HCS of $K_{v}-I$  exists if and only if
$\frac{v}{2}\equiv 1$ or $2\pmod 4$.
\end{thm}

In \cite{BMnew}, the authors study the case of a HCS of $K_v$ which is {\em both} cyclic and symmetric;
their result in the case $v$ even is that
there exists a cyclic and symmetric HCS of $K_v$ for all values for which a cyclic HCS exists, that is for
$\frac{v}{2}\equiv 1$ or $2\pmod4$ and $\frac{v}{2}$ not a prime power.

Very recently Michael Schroeder \cite{schr} studied HCS  for a graph $\Gamma$ in which each cycle is
fixed by a fixed-point free automorphism $\phi_n$ of $\Gamma$ of order $n >2$, so that $V(\Gamma)=mn$ for some $m$;
we shall  call such a HCS {\em $\phi_n$-symmetric}.\\
To admit a $\phi_n$-symmetric HCS, $\Gamma$ must be a subgraph of $K_{m\times n}$, and in \cite{schr}  the existence problem of a $\phi_n$-symmetric HCS for $K_{m\times n}$  is completely solved in the following result.

\begin{thm}[Schroeder \cite{schr}]\label{th:schr}
Let $m\ge2$ and $n\ge1$ be integers such that $(m-1)n$ is even.
A $\phi_n$-symmetric HCS for $K_{m\times n}$  exists if and only if, if  $n\equiv 2\pmod 4$ then $m\equiv 1$ or $2\pmod 4$.
\end{thm}

Note that we shall see the same non-existence condition later on in Corollary \ref{cor:ne}.
It makes sense therefore to study, as done in  \cite{BMnew} for the cocktail party graph, HCS for the complete multipartite graph that are {\em both} cyclic and symmetric.
As noted above, $\Kmn$ is the Cayley graph
$Cay[\Z_{mn}:\Z_{mn}-m\Z_{mn}]$, and we let $\phi_n$ be  the morphism $x\mapsto x+m \pmod{mn}$.
We have the following condition for a cycle in a cyclic cycle system to be $\phi_n$-invariant.

\begin{lem}\label{le:sym}
A cycle $C$ in a cyclic HCS of $K_{m\times n}$ is $\phi_n$-invariant if and only if $n$ divides $|Stab(C)|$ - or equivalently, if $|Orb(C)|$ divides $\frac{mn}{n}=m.$
\end{lem}

\begin{ex}\label{ex:10x6symsym}
Let us consider once more  the cycles we used in Example \ref{ex:10x6};
we can easily see  that the cycle system is also $\phi_{6}$-symmetric,
since the length of the orbit is $10$ for cycles $C_1$ and $C_2$, $2$ for cycles $C_3, C_4, C_5$ and $1$ for $C_6$.
\end{ex}

\section{Non-existence results}\label{se:ne}

In this section we shall present some non-existence results for cycle systems of the complete multipartite graph $\Kmn$;
the methods used here will be closely related to those used in \cite{BR},
where the case of the cocktail party graph is considered.
The results will concern general cycle systems; we will then apply these results to the hamiltonian case.

The following lemma is an immediate generalization of Lemma 2.1 of \cite{BR}, hence we omit the
proof.

\begin{lem}\label{LemmaNonEsiste}
Let $C=(c_0,c_1,\ldots,c_{\ell-1})$ be a cycle belonging to a cyclic cycle system of $K_{m\times n}$
and let $d$ be the order of $Stab(C)$. Then $Orb(C)$ is an $\ell$-cycle system of
$Cay[\Z_{mn}:\{\pm(c_{i-1}-c_{i})\ |\ 1\leq i \leq \frac{\ell}{d}\}]$.
\end{lem}

The next result adapts  Theorem 2.2 of \cite{BR} to the $m$-partite case.

\begin{prop}\label{PropParity}
Let $n$ be an even integer. The number of cycle-orbits of odd length in a cyclic cycle-decomposition of
$K_{m\times n}$ has the same parity of $\frac{m(m-1)n^2}{8}$.
\end{prop}

\begin{proof}
Let $\B$ be a cyclic cycle system of $K_{m\times n}$.
For every $\ell$-cycle $C=(c_0,c_1,\ldots,$ $c_{\ell-1})$ of $\B$ set
$$\sigma(C)=\sum_{i=1}^{\ell/d}(c_{i-1}-c_i)=(c_0-c_1)+(c_1-c_2)+\ldots+(c_{\ell/d-1}
-c_{\ell/d})=c_0-c_{\ell/d},$$
where $d$ is the order of $Stab(C)$.
It is easy to see that $c_{\ell/d}=c_0+\rho$ where $\rho$ is an element of $\Z_{mn}$ of order $d$ and hence we have
$$\sigma(C)=\frac{mnx}{d}\quad \textrm{with}\ \gcd(x,d)=1.$$
Since $n$ is even, we have that $\sigma(C)$ is even if and only if
 $d$ is a divisor of $\frac{mn}{2}$; on the other hand,
since the length of $Orb(C)$ is $\frac{mn}{d}$, also $|Orb(C)|$ is even if and only if
 $d$ is a divisor of $\frac{mn}{2}$. For any cycle $C\in\B$, we thus have that
  \begin{equation}\label{parity}
\sigma(C)\equiv |Orb(C)|\quad \pmod2.
 \end{equation}
 Let $\mathcal{S}=\{C_1,\ldots,C_s\}$ be a set of base blocks of $\B$, that is a complete
system of representatives for the orbits
 of the cycles of $\B$, so that we have
 $$\B=Orb(C_1)\cup Orb(C_2)\cup\ldots \cup Orb(C_s).$$
 By Lemma \ref{LemmaNonEsiste}, the cycles of $Orb(C_i)$ form a cycle system of
 $Cay[\Z_{mn}:\partial C_i]$ hence it results
$$Cay[\Z_{mn}:\Z_{mn}-m\Z_{mn}]=\ccup_{i=1}^{s}Cay[\Z_{mn}:\partial C_i]=Cay\left[\Z_{mn}:\partial \left(\ccup_{i=1}^{s} C_i\right)\right]$$
 so we obtain that
 \begin{equation}\label{partial}
\partial\left(\ccup_{i=1}^{s} C_i\right)=\Z_{mn}-m\Z_{mn}.
 \end{equation}
Note that $\partial C_i$ is a disjoint union of the set of summands of $\sigma(C_i)$ and the set of their opposites.
Hence, by (\ref{partial}), it follows that $\Z_{mn}-m\Z_{mn}$ is a disjoint union of the set of all summands of the sum $\sum_{i=1}^s\sigma(C_i)$
and the set of their opposites.
On the other hand, we have that $\Z_{mn}-m\Z_{mn}=\{\alpha m+1,\alpha m+2,\ldots, (\alpha+1)m-1\ |\ \alpha=0,1,2,\ldots,n-1\}$,
and so we can write
$$\sum_{i=1}^s \sigma(C_i)=s_1+s_2+\dots+s_{m-1}+s_{m+1}+\dots + s_{\frac{nm}{2}-1}$$
where $s_i=i$ or $-i$ for each $i$. So since $i$ and $ -i$ have the same parity,
it follows that
$$\sum_{i=1}^s\sigma(C_i)\equiv1+2+\dots + (m-1)+(m+1)+\dots +\left(\frac{nm}{2}-1\right)\equiv \frac{m(m-1)n^2}{8}\pmod 2.$$
From (\ref{parity}) we have
$$\sum_{i=1}^s|Orb(C_i)|\equiv \frac{m(m-1)n^2}{8}\pmod 2.$$
Hence the number of cycles $C_i$ of $\mathcal{S}$ whose orbit has odd length has the same
parity of $\frac{m(m-1)n^2}{8}$,
and the assertion follows.
\end{proof}

Now we are ready to prove the main non-existence result. In the following given a positive integer
$x$ by $|x|_2$
we will denote the largest $e$ for which $2^e$ divides $x$.

\begin{thm}
Let $n$ be an even integer. A cyclic $\ell$-cycle system of $\Kmn$ cannot exist
in each of the following cases:
\begin{itemize}
\item[(a)] $m\equiv 0\pmod 4$ and $|\ell|_2=|m|_2+2|n|_2-1$;
\item[(b)] $m\equiv 1\pmod 4$ and $|\ell|_2=|m-1|_2+2|n|_2-1$;
\item[(c)] $m\equiv 2,3\pmod 4$, $n\equiv 2\pmod 4$ and $\ell \not\equiv 0 \pmod 4$;
\item[(d)] $m\equiv 2,3\pmod 4$, $n\equiv 0\pmod 4$ and $|\ell|_2 = 2|n|_2$.
\end{itemize}
\end{thm}

\begin{proof}
Let $\B$ be an $\ell$-cycle system of $\Kmn$, obviously
$|\B|=|E(\Kmn)|/\ell=mn^2(m-1)/2\ell$. Hence the number of cycle-orbits of odd length of a cyclic
$\ell$-cycle system of $\Kmn$ has the same parity as
$mn^2(m-1)/2\ell$. By Proposition \ref{PropParity}, we have that
$mn^2(m-1)/2\ell\equiv mn^2(m-1)/8\pmod 2$.
Now the conclusion can be easily proved distinguishing four cases according to the congruence class of $m \pmod 4$.
\end{proof}

If  the cycles of the system are hamiltonian, that is if $\ell=mn$, we obtain the following
corollary.

\begin{cor}\label{cor:ne}
Let $n$ be an even integer. A cyclic HCS of $\Kmn$ cannot exist
if both $m\equiv 0,3 \pmod 4$ and $n\equiv 2\pmod 4$.
\end{cor}

\section{Existence of cyclic and symmetric HCS of $\Kmn$, $m$ even}\label{se:ex}

In this section we present direct constructions of cyclic and symmetric HCS
of the complete multipartite graph with an even number of parts.
Since $(m-1)n$ must be even, if $m$ is even then $n$ is  even too; the condition in Corollary \ref{cor:ne} tells us that, when
$n\equiv 2\pmod 4$, $m$ should also be congruent to $2\pmod 4$.  If these two requirements are met, we will show that
a cyclic and symmetric HCS of $\Kmn$ always exists, and therefore we will prove Theorem \ref{th:main}.

As observed in the Introduction, $K_{m\times 2}=K_{2m}-I$ is  the cocktail party
graph;  thus we can suppose $n\geq 2$, since for $n=2$ we can rely on  the following result.

\begin{thm}[Jordon, Morris \cite{JM}; Buratti, Merola \cite{BMnew}]\label{JM}
For an even integer $v\geq4$ there exists a cyclic and symmetric HCS of $K_v-I$
if, and only if, $v\equiv 2,4\pmod 8$ and $v\neq 2p^\alpha$ where $p$ is an odd prime and
$\alpha\geq 1$.
\end{thm}

We start by considering the complete bipartite graph.

\begin{prop}\label{prop:bipartito}
For any even integer $n$ there exists a cyclic and $\phi_n$-symmetric HCS of $K_{2\times n}$.
\end{prop}

\begin{proof}
Let $n=2\ell$; we need a set $\B$ of base cycles such that
$\partial\B=\pm\{1,3,\dots,2\ell-1\}$.
Let us first assume $\ell$ even. For $i=0,1,\dots,\ell/2-1$ consider the cycle $C_i=[0,4i+3]_2$; we have $\partial C_i=\pm\{4i+1,4i+3\}$, and thus $\B=\bigcup_{i=0}^{\frac{\ell}{2}-1}C_i$ is a set of hamiltonian cycles of $K_{2\times n}$ such that $\partial \B=\Z_{2n}-2\Z_{2n}$.
Now assume that $\ell$ is odd; for  $i=0,1,\dots,\lfloor\ell/2 \rfloor-1$ take $C_i=[0,4i+3]_2$ as above, and add the cycle $C'=[0]_{2\ell-1}$. Now $\B=(\bigcup_{i=0}^{\lfloor\ell/2 \rfloor-1}C_i)\cup C'$ is a set of base cycles for a HCS of $K_{2\times n}$.
This cycle system is also $\phi_n$-symmetric by Lemma \ref{le:sym}, since each cycle belongs to an orbit of length $1$ or $2$.
\end{proof}

\begin{ex}
In this example we construct a cyclic and $\phi_{14}$-symmetric HCS of $K_{2\times 14}$ following the
proof of Proposition \ref{prop:bipartito}.
Note that in this case $n=14$, hence $\ell=7$; so  we have the three cycles
\begin{eqnarray*}
C_0&=&(0,3,2,5,4,7,6,9,8,11,10,13,12,15,14,17,16,19,18,21,20,23,22,25,\\
&&24,27,26,1)\\
C_1&=&(0,7,2,9,4,11,6,13,8,15,10,17,12,19,14,21,16,23,18,25,20,27,22,1,\\
&&24,3,26,5)\\
C_2&=&(0,11,2,13,4,15,6,17,8,19,10,21,12,23,14,25,16,27,18,1,20,3,22,5,\\
&&24,7,26,9)
\end{eqnarray*}
which together with
\begin{eqnarray*}
C'&=&(0,13,26,11,24,9,22,7,20,5,18,3,16,1,14,27,12,25,10,23,8,21,6,19,4,\\
&&17,2,15)
\end{eqnarray*}
form a set of base cycles for a cyclic and $\phi_{14}$-symmetric HCS of $K_{2\times 14}$.
\end{ex}

Now we tackle the case $n\equiv 0\pmod 4$.

\begin{thm}\label{thm:m_pari}
Let $m$ be an even integer and $n\equiv 0 \pmod 4$. Then there exists a cyclic and $\phi_n$-symmetric HCS of $\Kmn$.
\end{thm}

\begin{proof}
We may assume $m\geq 4$, since if $m=2$, the statement follows from Proposition \ref{prop:bipartito}.
We shall first give a construction for $m$ a power of $2$. Let
$m=2^a$ and $n=4t$ with $a>1$ and $t\geq 1$.
Consider the following path:
\begin{eqnarray*}
P_{i,b} & = & [0,2mi+(2^{b+1}-1),1,2mi+(2^{b+1}-2),2,2mi+(2^{b+1}-3),\ldots,\\
&&  (2^{b-1}-1),2mi+(2^{b+1}-2^{b-1})]
\end{eqnarray*}
for all $b=1,\ldots,a$ and $i=0,\ldots,t-1$.
Observe that the elements of $P_{i,b}$ are pairwise distinct module $2^b$
and hence $A_{i,b}=[P_{i,b}]_{2^b}$ is a hamiltonian cycle of $\Kmn$.
One can check that
$$\partial A_{i,b}=\pm(\{2mi+2^{b-1}\} \cup \{2mi+(2^b+1),2mi+(2^b+2),\ldots,2mi+(2^{b+1}-1)\}).$$
It turns out that $ \partial (\cup A_{i,b})=\Z_{mn}-m\Z_{mn}$.
Hence the existence of a cyclic  HCS of $\Kmn$ follows
from Theorem \ref{thm:basecycles}.

Now assume
$m=2^a \m$ with $a\geq 1$ and $\m>1$ odd. Take $n=4t$ with  $t\geq 1$. We start constructing for all $i=0,\ldots,t-1$ the following paths:
\begin{equation}\label{pathP}
P_{i,j}=\left\{
\begin{array}{ll}
\left[0, 2m i+(4j-1) \right] & \textrm{if } j=1, \ldots, \frac{\m -1}{2} \\
\left[0, 2m i+(4j+1) \right] &  \textrm{if } j=\frac{\m+1}{2},\ldots,\m-1
\end{array}\right. ;
\end{equation}
\begin{eqnarray}\label{pathQ}
Q_{i,1} & = & [0, 2mi +(4\m -1),1, 2mi+(4\m-3),3, \ldots, \m-2, 2mi+3\m, \\
\nonumber &&\m+1,2mi+(3\m-1), \m+3, 2mi+(3\m-3),\ldots,2\m-2,\\
&& 2mi+(2\m+2)]; \nonumber
\end{eqnarray}
and, if $a\geq 2$, for all $b=2,\ldots,a$ take also
\begin{eqnarray*}
Q_{i,b} &= &[0, 2mi+(2^{b+1}\m-1), 1, 2mi+(2^{b+1}\m-2), 2,\ldots,(2^{b-1}\m-1),\\
&& 2mi+(2^{b+1}\m-2^{b-1}\m)].
\end{eqnarray*}
Observe that
since $2m i+(4j-1)$ and $2m i+(4j+1)$ are odd, $A_{i,j}=[P_{i,j}]_2$ is a hamiltonian cycle of $\Kmn$ for any $i,j$.
Furthermore, the elements of $Q_{i,b}$ are pairwise distinct module $2^b\m$ and hence $B_{i,b}=[Q_{i,b}]_{2^b\m}$ is a hamiltonian cycle of $\Kmn$ for any $i,b$.
One can check that the
$\partial A_{i,j}=\pm\{2mi+(4j-3),2mi+(4j-1)\}$ for $j=1,\ldots,\frac{\m-1}{2}$ and
$\partial A_{i,j}=\pm\{2mi+(4j-1),2mi+(4j+1)\}$ for $j=\frac{\m+1}{2},\ldots,\m-1$.
Hence for any fixed $i$ we have
$$\partial \left(\ccup_{j=1}^{\m-1} A_{i,j}\right)=
\pm \big(\{2mi+1,2mi+3,2mi+5,\ldots,2mi+(2\m-3)\}\cup$$
$$\{2mi+(2\m+1),2mi+(2\m+3),2mi+(2\m+5),\ldots,2mi+(4\m-3)\}\big).$$
Also,
\begin{eqnarray*}
\partial B_{i,1} & =& \pm (\{2mi+2,2mi+4,2mi+6,\ldots,2mi+2\m-2\} \cup \\
&&\{2mi+2\m+2,2mi+2\m+4,2mi+2\m+6,\ldots,2mi+4\m-2\} \cup \\
&&\{2mi+2\m-1,2mi+4\m-1\})
\end{eqnarray*}
and for $b=2,\ldots,a$
\begin{eqnarray*}
\partial B_{i,b} & =& \pm (\{2mi+2^{b-1}\m\} \cup \{2mi+(2^b\m+1),2mi+(2^b\m+2),\ldots,\\
&&2mi+(2^{b+1}\m-1)\}).
\end{eqnarray*}
It turn out that for every fixed $i$ we have
$$\partial \left(\ccup_{b=1}^a  B_{i,b}\right)=
\pm \big(\{2mi+2,2mi+4,2mi+6,\ldots,2mi+4\m\}\cup$$
$$\{2mi+(4\m+1),2mi+(4\m+2),2mi+(4\m+3),\ldots,2mi+(m-1)\}\cup$$
$$\{2mi+(m+1),2mi+(m+2),2mi+(m+3),\ldots,2mi+(2m-1)\}\big).$$
Let $\B=\left(\cup_{i,j} A_{i,j}\right)\cup \left(\cup_{i,b} B_{i,b}\right)$.
Hence, for what seen above, for every fixed $i$ we have
$$\partial\left(\ccup_{j=1}^{\m-1} A_{i,j} \right) \cup \partial \left(\ccup_{b=1}^a B_{i,b}\right)=
\pm \big(\{2mi+1,2mi+2,2mi+3,\ldots,2mi+(m-1)\}\cup$$
$$\{2mi+(m+1),2mi+(m+2),2mi+(m+3),\ldots,2mi+(2m-1)\}\big)$$
and so
$\partial \B=\Z_{mn}-m\Z_{mn}$. We conclude that $\B$ is a set of base cycles of a cyclic HCS of $\Kmn$.

It is easily seen from Lemma  \ref{le:sym} that these cycle systems are also $\phi_n$-symmetric, since in all cases the length of the orbit of each cycle divides $m$.
\end{proof}

\begin{ex}
Following the proof of Theorem \ref{thm:m_pari} we give here the construction of a set of base cycles of a cyclic
and $\phi_4$-symmetric HCS of $K_{18\times 4}$.
In the notation of the Theorem, $a=1$, $\m=9$ and $t=1$. Take the following cycles:
$$
A_{0,1}  =  [0,3]_2, \quad
A_{0,2}  =  [0,7]_2, \quad
A_{0,3}  =  [0,11]_2, \quad
A_{0,4}  =  [0,15]_2,
$$
$$
A_{0,5}  =  [0,21]_2, \quad
A_{0,6}  =  [0,25]_2, \quad
A_{0,7}  =  [0,29]_2, \quad
A_{0,8}  =  [0,33]_2,
$$
$$B_{0,1}=[0,35,1,33,3,31,5,29,7,27,10,26,12,24,14,22,16,20]_{18}.$$
We have
$$\partial \left(\ccup_{j=1}^8  A_{0,j}\right)=\pm\left(\{1,3,5,\dots,15\}\cup \{19,21,23,\ldots,33\}\right)$$
and
$$\partial B_{0,1}=\pm \left(\{2,4,6,\ldots,16\}\cup\{20,22,24,\ldots,34\}\cup \{17,35\}\right).$$
So, given $\B=\left(\cup_j A_{0,j}\right)\cup B_{0,1}$,
we have $\partial \B=\Z_{72}-18\Z_{72}$.
\medskip

Now, we give the construction of a set of base cycles of a cyclic and $\phi_8$-symmetric HCS of $K_{72\times 8}$.
Notice that $\m=9$ as before, but $a=3$ and $t=1$, so we need to construct a larger number of cycles. For $i=0$ we take
$$
A_{0,1}  =  [0,3]_2, \quad
A_{0,2}  =  [0,7]_2, \quad
A_{0,3}  =  [0,11]_2, \quad
A_{0,4}  =  [0,15]_2,
$$
$$
A_{0,5}  =  [0,21]_2, \quad
A_{0,6}  =  [0,25]_2, \quad
A_{0,7}  =  [0,29]_2, \quad
A_{0,8}  =  [0,33]_2,
$$
\begin{eqnarray*}
B_{0,1} & =& [0,35,1,33,3,31,5,29,7,27,10,26,12,24,14,22,16,20]_{18},\\
B_{0,2} & = & [0,71,1,70,2,69,3,68,\ldots,17,54]_{36},\\
B_{0,3} & =& [0,143,1,142, 2,141,3,140,\ldots,35,108]_{72}.
\end{eqnarray*}
We have
$$\partial\left(\ccup_{j=1}^8  A_{0,j}\right)\cup \partial\left(\ccup_{b=1}^3  B_{0,b}\right)=
\pm \left(\{1,2,3,\ldots,71\}\cup\{73,74,75,\ldots,143\}\right).$$
Furthermore, for $i=1$:
$$
A_{1,1}  =  [0,147]_2, \quad
A_{1,2}  =  [0,151]_2, \quad
A_{1,3}  =  [0,155]_2, \quad
A_{1,4}  =  [0,159]_2,
$$
$$
A_{1,5}  =  [0,165]_2, \quad
A_{1,6}  =  [0,169]_2, \quad
A_{1,7}  =  [0,173]_2, \quad
A_{1,8}  =  [0,177]_2,
$$
\begin{eqnarray*}
B_{1,1} & =& [0,179,1,177,3,175,5,173,7,171,10,170,12,168,14,166,16,164]_{18},\\
B_{1,2} & = & [0,215,1,214,2,213,3,212,\ldots,17,198]_{36},\\
B_{1,3} & =& [0,287,1,286, 2,285,3,284,\ldots,35,252]_{72}.
\end{eqnarray*}
We have
$$\partial\left(\ccup_{j=1}^8  A_{1,j}\right)\cup \partial \left(\ccup_{b=1}^3  B_{1,b}\right)=
\pm (\{145,146,147,\ldots,215\}\cup\{217,218,219,\ldots,287\}).$$
So, given $\B=\left(\cup_{i,j} A_{i,j}\right)\cup \left(\cup_{i,b} B_{i,b}\right)$,
we have $\partial \B=\Z_{576}-72\Z_{576}$.
\end{ex}

The following definition and lemmas are instrumental in proving Theorem \ref{thm:mn2}, where we shall
settle the case $n\equiv 2\pmod 4$.

\begin{defini}\label{defi:esse}
For all positive integers $s,d$ and all odd integers $w\geq 3$, set
$$S(s,d,w)=\left\{s,s+d,s+2d,s+3d,\ldots,s+\frac{w-3}{2}d\right\}$$
and 
$$\varphi(s,d,w)=\left|\{x \in S(s,d,w) : \gcd(x,w)=1\}\right|.$$
\end{defini}

\begin{lem}\label{lem:prime}
Assume $w=p^\alpha$ for some odd prime $p$.
Then
$\varphi(s,d,p^\alpha)=0$ if, and only if, either
$\gcd(s,d,p)>1$ or $w=3$ and  $3$ divides $s$.
\end{lem}

\begin{proof}
First, suppose that $\varphi(s,d,p^\alpha)=0$. Hence $p$ must divide
every element of $S(s,d,p^\alpha)$ and in particular $p$ divides $s$. If
$|S(s,d,p^\alpha)|=1$, then $p^\alpha=3$. In this case $w=3$ and $3$ divides $s$.
If $|S(s,d,p^\alpha)|>1$, the prime $p$ divides $s+d$ and so divides $d$.  It follows that $p$
divides $s,d$ and hence $\gcd(s,d,p)>1$.

Suppose now that $\gcd(s,d,p)>1$, namely that $p$ divides both $s$ and $d$.
Then $p$ divides every element of
$S(s,d,p^\alpha)$ and so
$\varphi(s,d,p^\alpha)=0$.
Finally, suppose $w=3$ and $s=3s_1$. Then $3$ divides the unique element of $S(s,d,3)=\{s\}$,
and hence $\varphi(s,d,3)=0$.
\end{proof}

\begin{lem}\label{lem:ok}
Assume that $\gcd(s,d,w)=1$ and $\gcd(3,s)=1$. Then
$\varphi(s,d,w)>0$.
\end{lem}

\begin{proof}
First, write $w=p_1^{\alpha_1}p_2^{\alpha_2}\cdots p_k^{\alpha_k}$ where
$3\leq p_1<p_2<\ldots<p_k$ are all primes.
We proceed by induction on $k$. Let $k=1$, that is $w=p_1^{\alpha_1}$.
By way of contradiction, suppose $\varphi(s,d,w)=0$, hence by  Lemma \ref{lem:prime}, we get either $\gcd(s,d,w)>1$ or $3$
divides $s$. In both cases, we get a contradiction with our hypothesis.\\
So, let $k>1$ and write $w=ap_k^{\alpha_k}$, where $a=p_1^{\alpha_1}p_2^{\alpha_2}\cdots
p_{k-1}^{\alpha_{k-1}}\geq 3$. By the induction hypothesis the set $S(s,d,a)$ contains an element
coprime to $a$. Let $s+id $ be this element, where $0\leq i\leq \frac{a-3}{2}$.
Now, the set $S(s+id, ad, p_k^{\alpha_k})$ contains only elements that are coprime to $a$.
If $\varphi(s+id,ad,p_k^{\alpha_k})=0$, then by Lemma \ref{lem:prime} we get $\gcd(s+id,ad,p_k)>1$
(since $p_k>3$).
Since $\gcd(a,p_k)=1$, it follows that $p_k$ divides $d$ and so divides $s$. We conclude that
$\gcd(s,d,w)>1$, in contrast with our hypothesis.
Hence, $\varphi(s+id,ad,p_k^{\alpha_k})>0$.
So there exists an element $x=s+id+jad$, for some $0\leq j\leq \frac{p_k^{\alpha_k}-3}{2}$,
that is coprime with both $a$ and $p_k$. Since $i+ja\leq \frac{w-3}{2}$, we obtain that $x\in
S(s,d,w)$, i.e.  $\varphi(s,d,w)>0$.
\end{proof}

\begin{thm}\label{thm:mn2}
Let $m,n$ be integers with $m,n\equiv 2 \pmod 4$.
Then there exists a cyclic and $\phi_n$-symmetric HCS of $\Kmn$.
\end{thm}

\begin{proof}
In view of Propositions \ref{prop:bipartito} and Theorem \ref{JM} we may assume
$m=2\m > 2$ and $n=4t+2> 2$.
Using the notation of Definition \ref{defi:esse} take
$$s=\left\{
\begin{array}{ll}
3\m+2 & \textrm{ if } m\equiv 2 \pmod 8\\
3\m -2 & \textrm{ if } m\equiv 6 \pmod 8\\
\end{array}
\right. ,$$
$d=4\m$ and $w=\frac{n}{2}$.
In view of Lemma \ref{lem:ok}, we obtain that the set $S(3\m\pm 2, 4\m,\frac{n}{2})$
contains an element $\nu=s+2m\kappa$  coprime
with $\frac{n}{2}$, where $0\leq \kappa \leq \frac{n-6}{4}$.
It is useful for the following to observe that $\gcd(\nu,mn)=1$, as $\gcd(3\m\pm 2, \m)=1$.

For all $i=0,\ldots,\kappa$ consider the paths $Q_{i,1}$ as in \eqref{pathQ}
and, if $\kappa\geq1$, for all $i=0,\ldots,\kappa-1$ consider the paths $P_{i,j}$ as in \eqref{pathP}.
If $t\geq \kappa+2$, for all $i=\kappa+1,\ldots,t-1$, take also the following paths:
\begin{eqnarray*}
\widetilde P_{i,j}& = &\left\{
\begin{array}{ll}
\left[0, (2i+1)m+(4j-1) \right] & \textrm{if } j=1, \ldots, \frac{\m -1}{2} \\
\left[0, (2i+1)m+(4j+1) \right] &  \textrm{if } j=\frac{\m+1}{2},\ldots,\m-1
\end{array}\right. ;\\
\widetilde Q_{i} & = & [0, (2i+1)m +(4\m -1),1, (2i+1)m+(4\m-3),3, \ldots, \m-2,\\
&&(2i+1)m+3\m, \m+1,(2i+1)m+(3\m-1), \m+3, \\
&&(2i+1)m+(3\m-3),\ldots,2\m-2, (2i+1)m+(2\m+2)].
\end{eqnarray*}
Also, define
$$u=\left\{
\begin{array}{ll}
\frac{3\m+1}{4} & \textrm{if } m\equiv 2 \pmod 8 \\[3pt]
\frac{3\m-1}{4} & \textrm{if } m\equiv 6 \pmod 8
\end{array}
\right.$$
and take the paths:
\begin{eqnarray*}
\widetilde R_{j}& = &\left\{
\begin{array}{ll}
\left[0, 2m\kappa+(4j-1) \right] & \textrm{if } j=1, \ldots, \frac{\m -1}{2} \\
\left[0, 2m\kappa+(4j+1) \right] &  \textrm{if } j=\frac{\m+1}{2},\ldots,\m-1 \textrm{ and  } j\neq u
\end{array}\right. ;\\
\widetilde S & = & [0, (2\kappa+1)m +(4\m -1),1,(2\kappa+1)m +(4\m -2),2,\ldots,\\
&& \m-1,(2\kappa+1)m +3\m ].
\end{eqnarray*}
As we have seen in Theorem \ref{thm:m_pari}, $A_{i,j}=[P_{i,j}]_2$ and $B_{i}=[Q_{i,1}]_{m}$ are hamiltonian cycles of $\Kmn$ for any $i,j$.
Further, consider $C_{i,j}=[\widetilde P_{i,j}]_2$, $D_i=[\widetilde Q_{i}]_{m}$, $E_j=[\widetilde R_j]_2$, $F=[\widetilde S]_{m}$ and $G=[0]_\nu$.  It easy to see that they are all hamiltonian cycles of $\Kmn$ and that:
for $j=1,\ldots,\frac{\m-1}{2}$
$$\partial C_{i,j}=\pm\{(2i+1)m+(4j-3),(2i+1)m+(4j-1)\}$$
and for $j=\frac{\m+1}{2},\ldots,\m-1$
$$\partial C_{i,j}=\pm\{(2i+1)m+(4j-1), (2i+1)m+(4j+1)\}.$$
Also,
\begin{eqnarray*}
\partial D_{i} & =& \pm (\{(2i+1)m+2, (2i+1)m+4,(2i+1)m+6,\ldots, (2i+1)m+\\
&&+2\m-2\} \cup \{(2i+1)m+2\m+2,(2i+1)m+2\m+4,\ldots,(2i+1)m+\\
&&+ 4\m-2\} \cup \{(2i+1)m+2\m-1,(2i+1)m+4\m-1\});
\end{eqnarray*}
moreover, for $j=1,\ldots,\frac{\m-1}{2}$
$$\partial E_{j}=\pm\{2m\kappa+(4j-3),2m\kappa+(4j-1)\}$$
and for $j=\frac{\m+1}{2},\ldots,\m-1$ with $j\neq u$
$$\partial E_{j}=\pm\{2m\kappa+(4j-1), 2m\kappa+(4j+1)\}.$$
Finally,
$$\partial F=\pm (\{(2\kappa+2)m+1,(2\kappa+2)m+2,\ldots, (2\kappa+3)m-1 \}\cup \{2m\kappa+3\m\} )$$
and $\partial G=\pm \{\nu\}$.

Let $\B=\left(\cup_{i,j} A_{i,j}\right)\cup \left(\cup_{i} B_{i}\right)\cup\left(\cup_{i,j} C_{i,j}\right)\cup \left(\cup_{i} D_{i}\right)
\cup \left(\cup_{j} E_{j}\right)\cup F\cup G$.
It is routine to check that $\partial \B=\Z_{mn}-m\Z_{mn}$,
hence we conclude that $\B$ is a set of base cycles of a cyclic HCS of $\Kmn$.
Once more, it is easily checked using Lemma  \ref{le:sym} that this cycle system is
also $\phi_n$-symmetric, since in all cases the length of the orbit of each cycle divides $m$.
\end{proof}

We point out that the base cycles used in Example \ref{ex:10x6} were constructed following the proof of Theorem \ref{thm:mn2}. In particular, according to the notation of the Theorem we have
$$C_1=B_0,\quad C_2=F,\quad C_3=E_1,\quad C_4=E_2,\quad C_5=E_3, \quad C_6=G.$$

\begin{ex}
Here we present a set of base cycles of a cyclic
and $\phi_{14}$-symmetric HCS of $K_{6\times 14}$.
In the notation of Theorem \ref{thm:mn2}, $\m=t=3$ and we choose $\kappa=1$ and $\nu=19$ which is coprime with $6\cdot 14$.
Following the proof of the Theorem we have to take the following cycles:
$$A_{0,1}  =  [0,3]_2, \quad A_{0,2}  =  [0,9]_2, \quad
B_{0}=[0,11,1,9,4,8]_{6},\quad B_{1}=[0,23,1,21,4,20]_{6},$$
$$C_{2,1}  =  [0,33]_2, \quad C_{2,2}  =  [0,39]_2, \quad
D_{2}=[0,41,1,39,4,38]_{6}, \quad
E_1=[0,15]_2,$$
$$F=[0,29,1,28,2,27]_6,\quad G=[0]_{19}.$$
It results
$$\partial \left(\ccup_{i=1}^{2}  A_{0,i}\right)=\pm\{1,3,7,9\},\;\; \partial \left(\ccup_{i=0}^{1} B_{i}\right)=\pm\{2,4,5,8,10,11,14,16,17,20,22,23\},$$
$$\partial \left(\ccup_{j=1}^{2} C_{2,j}\right)=\pm\{31,33,37,39\},\quad \partial D_2=\{32,34,35,38,40,41\},$$
$$\quad\partial E_1=\pm\{13,15\},\quad \partial F=\pm\{21,25,26,27,28,29\},\quad \partial G=\pm\{19\}.$$
So, called $\B$ the union of the constructed cycles, we have $\partial \B=\Z_{84}-6\Z_{84}$.
\smallskip

Now, we give a set of base cycles of a cyclic
 and $\phi_{10}$-symmetric HCS of $K_{10\times 10}$.
 In the notation of Theorem \ref{thm:mn2},
 $\m=5$, $t=2$ and we choose $\kappa=1$ and $\nu=37$ which is coprime with $100$.
 We have to take the following cycles:
$$A_{0,1}=[0,3]_2,\quad A_{0,2}=[0,7]_2,\quad A_{0,3}=[0,13]_2,\quad A_{0,4}=[0,17]_2,$$
$$B_{0}=[0,19,1,17,3,15,6,14,8,12]_{10},\quad
B_{1}=[0,39,1,37,3,35,6,34,8,32]_{10},$$
$$
E_1=[0,23]_2,\quad E_2=[0,27]_2,\quad E_3=[0,33]_2,$$
$$F=[0,49,1,48,2,47,3,46,4,45]_{10},\quad G=[0]_{37}.$$
We have:
$$\partial \left(\ccup_{i=1}^{4} A_{0,i}\right)=\pm\{1,3,5,7,11,13,15,17\},$$
$$\partial \left(\ccup_{i=0}^{1} B_{i}\right)=\pm\{2,4,6,8,9,12,14,16,18,19,22,24,26,28,29,32,34,36,38,39\},$$
$$\partial \left(\ccup_{j=1}^{3} E_{j}\right)=\pm\{21,23,25,27,31,33\},$$
$$\partial F=\pm\{35,41,42,43,44,45,46,47,48,49\}, \quad \partial G=\pm\{37\}.$$
Hence, called $\B$ the union of the constructed cycles, we have $\partial \B=\Z_{100}-10\Z_{100}.$
\end{ex}

\end{document}